\documentclass{amsart}

\usepackage{amssymb}
\usepackage{mathrsfs}

\newtheorem{thm}{Theorem}[section]

\newtheorem{lem}[thm]{Lemma}
\newtheorem*{lem*}{Lemma}

\theoremstyle{remark}

\theoremstyle{definition}

\newcounter{my_enumerate_counter}

\newcommand\comment[1]{}

\renewcommand{\epsilon}{\varepsilon}

\title[von Neumann-Day problem]
{A geometric solution to the von {N}eumann-Day 
problem for finitely presented groups}

\keywords{space of marked groups, left orderable, free group, piecewise, projective, torsion free}

\subjclass[2010]{Primary: 43A07; Secondary: 20F05}

\title[Approximating the free group by groups of homeomorphisms]
{Approximating nonabelian free groups by groups of homeomorphisms of the real line.}

\subjclass[2010]{Primary: 43A07; Secondary: 20F05}

\thanks{
The author thanks Matt Brin and Sang-hyun Kim for comments on the exposition.
This research has been supported by a Swiss national science foundation ``Ambizione" grant.}

\author{Yash Lodha}

\address{EPFL
\\ Lausanne\\ Switzerland}

\email{{\tt yash.lodha@epfl.ch}}

\begin{document}

\begin{abstract}
We show that for a large class $\mathcal{C}$ of finitely generated groups of orientation preserving homeomorphisms of the real line, the following holds: 
Given a group $G$ of rank $k$ in $\mathcal{C}$, there is a sequence of $k$-markings $(G, S_n), n\in \mathbf{N}$ whose limit in the space of
marked groups is the free group of rank $k$ with the standard marking.
The class we consider consists of groups that admit actions satisfying mild dynamical conditions and a certain ``self-similarity" type hypothesis.
Examples include Thompson's group $F$, Higman-Thompson groups, Stein-Thompson groups, various Bieri-Strebel groups, the golden ratio Thompson group,
and finitely presented nonamenable groups of piecewise projective homeomorphisms.
For the case of Thompson's group $F$ we provide a new and considerably simpler proof of this fact proved by Brin in \cite{Brin}.
\end{abstract}

\maketitle

\section{Introduction}

The group $\textup{Homeo}^+(\mathbf{R})$ is the group of orientation preserving homeomorphisms of $\mathbf{R}$.
The study of this group and its subgroups is a topic that has received attention in recent years (See \cite{DeroinNavasRivas} and \cite{NavasICM}, for instance).
In this article we consider finitely generated subgroups of $\textup{Homeo}^+(\mathbf{R})$ that satisfy certain natural dynamical conditions,
and their markings in \emph{the space of marked groups}:
The term refers to a family of spaces, and each individual space is denoted as $\mathcal{G}_m$, for $m\in \mathbf{N}$. 
This consists of pairs $(G,S)$ where $G$ is a finitely generated group and $S$ is a finite ordered generating set for $G$ of cardinality $m$.
These are considered up to the following notion of equivalence:
$(G_1,S_1)$ and $(G_2,S_2)$ are equivalent if there is an isomorphism $G_1\to G_2$ that induces an order preserving bijection $S_1\to S_2$.

The set $\mathcal{G}_m$ is equipped with a natural metric. 
The marked groups $(G_1,S_1)$ and $(G_2,S_2)$ are $e^{-n}$ apart if $n$ is the largest number such that the $n$-balls in the respective marked Cayley graphs are isomorphic.
This makes $\mathcal{G}_m$ a totally disconnected compact metrizable space.
This space was considered by Gromov in \cite{Gromov} in his celebrated work on groups of polynomial growth.
The space was first systematically studied by Grigorchuk in \cite{Grigorchuk}, and has an antecedent in the Chabauty topology (see \cite{Bridson}).

It is natural to inquire about the closure properties of certain families of groups (with varying markings) in this space.
This can be done for certain classes (such as free groups or hyperbolic groups, see \cite{ChampetierGuirardel} and \cite{Champetier}),
or for an individual group with varying markings.
In this article we consider the latter.
We consider group actions $G<\textup{Homeo}^+(\mathbf{R})$ that satisfy the following:
\begin{enumerate}
\item The action is minimal, i.e. the orbits are dense.
\item The groups of germs at $\pm\infty$ are conjugate to a subgroup of the affine group.
\item There exist elements $f,g$ and real numbers $r_1<r_2$ such that for each $x\in (-\infty, r_1), y\in (r_2,\infty)$ we have $$x\cdot f=x\qquad y\cdot f\neq y\qquad x\cdot g\neq x\qquad y\cdot g=y$$ 
\item There is a compact interval $I$ such that the subgroup of elements that fix each point in $\mathbf{R}\setminus I$ is finitely generated.
\end{enumerate}

The class of groups that admit actions satisfying the above shall be denoted by $\mathcal{C}$.
Groups in this family include Thompson's group $F$ (\cite{CFP}), Higman-Thompson groups $F_n$, 
Stein-Thompson groups (\cite{Stein}), various Bieri-Strebel groups (\cite{BieriStrebel}), the finitely presented nonamenable groups of piecewise projective homeomorphisms constructed by the author with Justin Moore (\cite{LodhaMoore}),
and Cleary's ``golden ratio" Thompson group (\cite{BRN}) among various other examples.
Note that the standard actions of many of these examples are prescribed as homeomorphisms of a compact interval.
However, these actions are easily seen to be topologically conjugate to actions on the real line that satisfy the above. 
Recall that the \emph{rank} of a finitely generating group is the smallest number $k\in \mathbf{N}$ such that the group admits a generating set of size $k$.
We prove the following.

\begin{thm}\label{Main}
Let $G\in \mathcal{C}$ be a group of rank $k$.
Then there is a sequence of $k$-markings $(S_n)_{n\in \mathbf{N}}$ of $G$ such that the nonabelian free group $\mathbf{F}_k$ with the standard marking is the limit of $(G,S_n)_{n\in \mathbf{N}}$ in $\mathcal{G}_k$.
\end{thm}

We remark that for the groups we consider, it is straightforward to find sequences of marked subgroups that approximate the free group in this sense
(see Lemma \ref{alpha}). 
However, the core difficulty in establishing Theorem \ref{Main} lies in finding the appropriate markings of the group itself.

Also note that as a consequence of Theorem \ref{Main} it follows that there is a sequence of $2$-markings of Thompson's group $F$ which limits to the standard marking of $\mathbf{F}_2$.
This was proved by Brin in \cite{Brin}. (For a related result, see \cite{AST}). 
Brin's proof is technically demanding, and although self contained, parts of it are obtained by computer assisted calculations. 
We provide a new and considerably simpler proof of this fact. 
For the convenience of the reader we prove this separately in Section $4$. 

Two key concepts play an important role in our proofs.
The first is the fact that the groups we consider do not satisfy a law, since they contain Thompson's group $F$ as a subgroup.
(For a proof that $F$ does not satisfy a law, we refer the reader to \cite{BrinSq}.)
The second is part $(4)$ of the definition of $\mathcal{C}$.
The finitely generated subgroup prescribed by this condition is sufficiently rich that it can supply carefully constructed elements.
Using these elements, we modify the generators of the ambient group so it has a cayley graph with the feature that the $n$-ball is isomorphic to the $n$-ball of the standard cayley graph of the free group of the same rank.
\section{Preliminaries}

All actions in this paper will be right actions.
Let $I=\mathbf{R}$ or a compact subinterval.
Given an element $g\in \textup{Homeo}^+(I)$, the \emph{support} of $g$, or $Supp(g)$, is defined as the set $\{x\in I\mid x\cdot g\neq x\}$.
The set of fixed points of $g$ that are accumulation points of $Supp(g)$ shall be denoted by $Tran(g)$, or the \emph{set of transition points of }$g$.




If $G<\textup{Homeo}^+(\mathbf{R})$, then we define the \emph{groups of germs} at $\pm \infty$ as follows.
First we define the groups $$G_{\infty}^+=\{g\in G\mid \exists r\in \mathbf{R}, g\text{ fixes each point in }(r,\infty)\}$$ $$G_{\infty}^-=\{g\in G\mid \exists r\in \mathbf{R}, g\text{ fixes each point in }(-\infty,r)\}$$
The group of germs at $\infty$ is $G/G_{\infty}^+$ and the group of germs at $-\infty$ is $G/G_{\infty}^-$.

Recall that in \cite{LodhaCoherent} we had defined the following notion.
We define a group action $G<\textup{Homeo}^+(\mathbf{R})$ to be \emph{coherent} if:
\begin{enumerate}
\item The action is minimal, i.e. the orbits are dense.
\item The groups of germs at $\pm\infty$ are solvable.
\item There exist elements $f,g$ and real numbers $r_1<r_2$ such that for each $x\in (-\infty, r_1), y\in (r_2,\infty)$ we have $$x\cdot f=x\qquad y\cdot f\neq y\qquad x\cdot g\neq x\qquad y\cdot g=y$$ 
\end{enumerate}

Hence the group actions described in the introduction that prescribe the definition of $\mathcal{C}$ satisfy the definition of a Coherent action.
We recall the following Lemmas that were proved in \cite{LodhaCoherent}.

\begin{lem}\label{F}
Let $G<\textup{Homeo}^+(\mathbf{R})$ be a coherent group action and let $F$ denote Thompson's group $F$.
Then there is a compact interval $J$ and elements $f,g\in G$ such that $$\langle f,g\rangle\cong F\qquad Supp(f),Supp(g)\subset J$$
\end{lem}

We remark that the above was shown in part $(1)$ of Theorem $1.1$ in \cite{LodhaCoherent} without the additional condition involving the compact interval $J$.
To see that there is a copy of $F$ supported on a compact $J$, first assume that there is a copy of $F$ in $G$ supported on a non compact interval. 
Then recall that $F<F'=F''$, and that the groups of germs of $G$ at $\pm \infty$ are solvable. 

\begin{lem}\label{transitivity}
Let $G<\textup{Homeo}^+(\mathbf{R})$ be a coherent group action.
Then the following holds.
\begin{enumerate}
\item Let $U_1,U_2$ be compact intervals with nonempty interior in $\mathbf{R}$.
Then there is a $g\in G$ such that $U_1\cdot g\subset U_2$.
\item Let $U_1,U_2,V_1,V_2$ be compact intervals with nonempty interior in $\mathbf{R}$ such that $$sup(U_1)<inf(U_2)\qquad sup(V_1)<inf(V_2)$$
Then there is a $g\in G$ such that $U_1\cdot g\subset V_1$ and $U_2\cdot g\subset V_2$.
\end{enumerate}
\end{lem}

For proofs we refer the reader to Lemmas $3.4$ and $4.4$ in \cite{LodhaCoherent}.

\section{Markings on groups in the class $\mathcal{C}$}

In this section we shall prove Theorem \ref{Main}.
Let $G<\textup{Homeo}^+(\mathbf{R})$ be a group action that satisfies conditions $(1)-(4)$ prescribed in the introduction.
We also refer to the underlying group as $G$, which we assume has rank $k$.
Note that since topological conjugation does not affect the conditions, we can assume that the interval $I$ that witnesses condition $(4)$
is $[0,1]$.
Let $S=\{f_1,...,f_k\}\subset G$ be a generating set for $G$.
Let $B_n$ denote the $n$-ball of the Cayley graph of $(G,S)$.
The subgroup of $G$ consisting of elements that fix each point in $\mathbf{R}\setminus [0,1]$ is denoted by $H$ and we fix a generating set $\{h_1,...,h_m\}$ for it.
Let $$h_i=u_1^{(i)}...u_{l_i}^{(i)}\qquad 1\leq i\leq m$$ be words in $S$ (each $u_j^{(i)}\in S\cup S^{-1}$).

\begin{lem}\label{alpha}
For each $n\in \mathbf{N}$, there are elements $\alpha_1^{(n)},...,\alpha_k^{(n)}\in H'$ such that the following holds.
The $n$-ball of the Cayley graph of the group $\langle \alpha_1^{(n)},...,\alpha_k^{(n)}\rangle$ with the given marking
is isomorphic to the $n$-ball of $\mathbf{F}_k$ with the standard marking. 
\end{lem}

\begin{proof}
Recall from Lemma \ref{F} that there is a compact interval $J$ such that there are elements $f,g\in G$
whose supports lie in $J$ and which generate a copy of Thompson's group $F$.
From Lemma \ref{transitivity}, we find $h\in G$ such that $J\cdot h\subset [0,1]$. 
Then $h^{-1} f h, h^{-1} g h\in H$ generate a copy of $F$.
Recall that Thompson's group $F$ does not satisfy a law.
It follows that the infinite direct product of copies of $F$ contains nonabelian free subgroups of each rank.
Since the direct sum of any finite number of copies of $F$ embeds in $F'$, there are elements $\alpha_1^{(n)},...,\alpha_k^{(n)}\in F'< H'$ that have the desired property.
\end{proof}


The existence of certain \emph{special elements} in our group action will be crucial in our proof.
We say that an element $f\in \textup{Homeo}^+(\mathbf{R})$ is $J$-\emph{special} for a compact interval $J=[r_1,r_2]$ if:
\begin{enumerate}
\item $x\cdot f=x$ for each $x\in (-\infty, r_1]$.
\item $x\cdot f>x$ for each $x\in (r_2,\infty)$.
\end{enumerate}

\begin{lem}\label{specialelements}
For each compact interval $I\subset \mathbf{R}$ there is an $I$-special element $f\in G$.
\end{lem}

\begin{proof}
Such an element for at least one compact interval $I_1$ exists in $G$ thanks to condition $(3)$ prescribed in the introduction.
Let $g$ be a $I_1$-special element.
Using Lemma \ref{transitivity}, we find an element $g_1\in G$ such that $I_1\cdot g_1\subset I$.
Then $g_1^{-1} g g_1$ is the required $I$-special element.
\end{proof}

Throughout the rest of the section we fix a $[0,\frac{1}{2}]$-special element $\tau_1\in G$ and a $[1,\frac{3}{2}]$-special element $\tau_2\in G$.
We also refer to a choice of $S$-words representing $\tau_1,\tau_2$ as $\tau_1,\tau_2$ respectively.

A set $I\subset \mathbf{R}$ which is a union of finitely many pairwise disjoint compact intervals is said to be \emph{$\tau$-compatible},
if for any choice of homeomorphisms $p_1,...,p_k$ with supports contained in $I$, the following holds.
Let $\tau_1',\tau_2'$ be the homeomorphisms prescribed by the words obtained by replacing each occurrence of $f_i^{\pm 1}\in S\cup S^{-1}$ in $\tau_1,\tau_2$ by $(p_if_i)^{\pm 1}$ for $1\leq i\leq k$.
Then $\tau_1'$ is also $[0,\frac{1}{2}]$-special and $\tau_2'$ is also $[1,\frac{3}{2}]$-special.

The following is a straightforward exercise involving uniform continuity and compactness, and we leave it for the reader.
Note that one may choose $M_1>>2$ here.

\begin{lem}\label{tiny}
There is an $M_1\in \mathbf{R}$ such that for each compact interval $J\subset (M_1,\infty)$ there is an $\epsilon>0$ for which the following holds.
Let $I\subset J$ be a union of finitely many compact intervals that are pairwise disjoint and of length at most $\epsilon$.
Then $I$ is $\tau$-compatible.
\end{lem}

Fix an $M_2\in \mathbf{R}$ such that for each $1\leq i\leq m, 1\leq j\leq l_i$ the following holds:
\begin{enumerate}
\item $([M_2-1,\infty)\cdot u_1^{(i)}...u_j^{(i)})\cap [0,2]=\emptyset$.
\item $([0,2]\cdot u_1^{(i)}...u_j^{(i)})\cap [M_2-1,\infty)=\emptyset$.
\end{enumerate}
Recall that $B_n$ is the ball of radius $n$ in $(G,S)$.
We say that a compact interval $J$ is \emph{$n$-admissible} if:
\begin{enumerate}
\item The intervals in the set $\mathcal{X}=\{J\cdot f\mid f\in B_n\}$ are pairwise disjoint and lie in $(M_2,\infty)$.

\item The union of the intervals in $\mathcal{X}$ is $\tau$-compatible.

\item For each $f,g\in B_n$ and $J\in \mathcal{X}$, if $J\cdot f\cap J\cdot g\neq \emptyset$
then $f\restriction J=g\restriction J$.
\end{enumerate}

\begin{lem}\label{admissible}
For each $n\in \mathbf{N}$, there exists a compact interval that is $n$-admissible.
\end{lem}

\begin{proof}

Recall that the group of germs of $G$ at $\infty$ is topologically conjugate to a subgroup of the affine group,
and that each non-trivial affine transformation of the real line has at most one fixed point in the real line.
(Recall from the preliminaries that $Tran(f)$ is the set of transition points of $f$.)
Hence the following is defined $$M_3=Sup\{Tran(f)\mid f\in B_{2n}\}<\infty$$

Let $x\in \mathbf{R}$ be such that $$X=\{x\cdot f\mid f\in B_n\}\subset (M_1,\infty)\cap (M_2,\infty)\cap (M_3,\infty)$$
Consider $x_1,x_2\in X$, and two distinct elements $f,g\in B_n$ such that $$x_1\cdot f=x_1\cdot g=x_2$$
Then $fg^{-1}$ fixes $x_1$. 
Since $fg^{-1}\in B_{2n}$ from our hypothesis above it follows that $fg^{-1}$ must be the identity on a neighbourhood of $x_1$.
Hence the restriction of $f,g$ on a suitable neighbourhood of $x_1$ is the same. 

Condition $(3)$ needs to be guaranteed for finitely many elements and intervals. 
Hence using the above observation this can be settled by choosing a sufficiently small compact interval containing $x$.
Conditions $(1),(2)$ can be further guaranteed by choosing an even smaller interval if needed.
For $(2)$ one applies Lemma \ref{tiny} to the compact interval $[inf(X),sup(X)]$.
It follows that a sufficiently small compact interval containing $x$ is $n$-admissible.
\end{proof}

Now we shall conclude the proof of Theorem \ref{Main}.
Using Lemma \ref{admissible}, for each $n\in \mathbf{N}$ we find a compact interval $I_n\subset \mathbf{R}$ that is $n$-admissible.
Using Lemma \ref{transitivity},
we can find an element $g\in G$ such that $[0,1]\cdot g\subset I_n$.
For each $1\leq i\leq k$, let $$\beta_i^{(n)}=g^{-1} \alpha_i^{(n)} g$$

We define the following equivalence relation on $B_n$.
For $f,g\in B_n$, $f\equiv g$ if $f\restriction I_n=g\restriction I_n$.
We choose a set of representatives for the equivalence classes and denote it by $\widetilde{B_n}\subset B_n$.

For each $1\leq i\leq k$, we define 
$$p_i^{(n)}=\prod_{f\in \widetilde{B_n}} f^{-1} \beta_i^{(n)} f\qquad s_i^{(n)}=p_i^{(n)} f_i$$
Let $G_n$ be the group generated by $S_n=\{s_1^{(n)},...,s_k^{(n)}\}$.

{\bf Claim}:
\begin{enumerate}
\item The $n$-ball of the Cayley graph of $(G_n,S_n)$ is isomorphic to the $n$-ball of the free group $\mathbf{F}_k$
with the standard marking.
\item $G_n=G$.
\end{enumerate}

The first part of the claim is clear from the construction.
The restriction of the action of any freely reduced $S_n$-word of length at most $n$ to the interval $I_n$ will be non-trivial.
We shall prove the second part.
Note that $G_n$ is a subgroup of $G$ by definition.
It suffices to show that $p_1^{(n)},...,p_k^{(n)}\in G_n$.
This would imply that $f_1,...,f_k\in G_n$ and hence $G_n=G$.

Let $\nu_1,\nu_2$ be the homeomorphisms prescribed by the words obtained by replacing each occurrence of $f_i^{\pm 1}\in S\cup S^{-1}$ in $\tau_1,\tau_2$ by $(p_i^{(n)}f_i)^{\pm 1}$ for $1\leq i\leq k$.
Since $I_n$ is $n$-admissible (and in particular the union of the intervals in the set $\{I_n\cdot f\mid f\in B_n\}$ is $\tau$-compatible), it follows that $\nu_1$ is $[0,\frac{1}{2}]$-special and $\nu_2$ is $[1,\frac{3}{2}]$-special.
Clearly, $\nu_1,\nu_2\in G_n$.

Note that for a sufficiently large $l\in \mathbf{N}$, $$\nu_1^{l} p_i^{(n)} \nu_1^{-l}\in H$$ for each $1\leq i\leq k$.

{\bf Subclaim 1}: $\nu_1^{l} p_i^{(n)} \nu_1^{-l}\in H'$.

It suffices to show this for $\nu_1^l \beta_i^{(n)} \nu_1^{-l}$, since the same argument applies to each of the conjugates of $\beta_i^{(n)}$ whose product equals $p_i^{(n)}$.
Recall that we chose $\alpha_i^{(n)}\in H'$.
Also recall that we chose an element $g\in G$ such that $[0,1]\cdot g\subset I_n$ and
for each $1\leq i\leq k$ we had let $$\beta_i^{(n)}=g^{-1} \alpha_i^{(n)} g$$
Since $$[0,1]\cdot g \nu_1^{-l}\subset [0,1]$$ it follows that $$\nu_1^l g^{-1} H g \nu_1^{-l}\subseteq H$$
Therefore $$\nu_1^l g^{-1} H' g \nu_1^{-l}\subseteq H'\qquad \nu_1^l \beta_i^{(n)} \nu_1^{-l}\in H'$$
This proves the claim.
Thanks to Subclaim $1$, we reduce the proof that $p_1^{(n)},..., p_k^{(n)}\in G_n$ to the following.

{\bf Subclaim 2}: $H'\subset G_n$. 

Recall that $h_1,...,h_m$
are words in $S$ that generate the subgroup $H$.
Let $\zeta_1,...,\zeta_m\in G_n$ be the elements prescribed by the words obtained by replacing each occurrence of $f_i^{\pm 1}\in S\cup S^{-1}$ in $h_1,...,h_m$ by $(p_if_i)^{\pm 1}$ for $1\leq i\leq k$.
By construction, since $I_n$ is $n$-admissible, the following is immediate for each $1\leq i\leq m$:
\begin{enumerate}
\item $\zeta_i\restriction [0,1]=h_i\restriction [0,1]$.
\item There is a compact interval $L\subset (2,\infty)$ such that $Supp(\zeta_i)\subset [0,1]\cup L$.
\end{enumerate}
Consider a pair of elements $g_1,g_2\in H$.
Let $\lambda_1,\lambda_2\in G_n$ be the elements prescribed by replacing the letters $h_1^{\pm 1},...,h_m^{\pm 1}$ by $\zeta_1^{\pm 1},...,\zeta_m^{\pm 1}$ in the respective words representing $g_1,g_2$.
It follows that:
\begin{enumerate}
\item $\lambda_1\restriction [0,1]=g_1\restriction [0,1]$ and $\lambda_2\restriction [0,1]=g_2\restriction [0,1]$.
\item There is a compact interval $L\subset (2,\infty)$ such that $$Supp(\lambda_1)\subset [0,1]\cup L\qquad Supp(\lambda_2)\subset [0,1]\cup L$$
\end{enumerate}
For a sufficiently large $l\in \mathbf{N}$, we obtain  $$[g_1,g_2]=[\nu_2^l \lambda_1 \nu_2^{-l},  \lambda_2]\in G_n$$
Since $g_1,g_2$ were arbitrary elements of $H$, we obtain that $H'\subset G_n$.
This finishes the proof.

\section{The special case of Thompson's group $F$}
In this section we provide a short proof for the special case of Thompson's group $F$ for the convenience of the reader.
We denote the free group of rank $2$ by $\mathbf{F}_2$, and Thompson's group $F$ by $F$ .
Recall that $F$ is isomorphic to the group of piecewise linear homeomorphisms of $\mathbf{R}$ generated by 
$$f_1(t)=t+1\qquad
f_2(t)=
\begin{cases}
 t&\text{ if }t\leq 0\\
 2t&\text{ if }0\leq t\leq 1\\
 t+1&\text{ if }1\leq t\\
\end{cases}
$$

In \cite{Brin}, the following was shown.

\begin{thm}
(Brin) The free group of rank $2$ is a limit of $2$-markings of $F$.
\end{thm}

We provide a new proof that avoids the considerable computations in \cite{Brin}.

\begin{proof}
We consider two subgroups $F_1,F_2\leq \textup{Homeo}^+(\mathbf{R})$, both isomorphic to $F$.
$F_1$ is the group generated by $f_1,f_2$ defined above.
$F_2$ is the subgroup of $F_1$ consisting of all homeomorphisms in $F_1$ that pointwise fix $\mathbf{R}\setminus [0,1]$.
Note that the restriction of $F_2$ to $[0,1]$ is precisely the standard copy of $F$ in $\textbf{PL}^+[0,1]$.

Recall that Thompson's group $F$ does not satisfy a law.
It follows that the free group of rank $2$ embeds in an infinite direct product of copies of $F$.
Since each finite direct sum of copies of $F$ embeds in $F'$, the following is immediate.
There is a sequence $(a_n,b_n), n\in \mathbf{N}$ of pairs of elements of $F'$ such that
the $n$-ball of the Cayley graph of $\langle a_n,b_n\rangle$ with respect to the generating set $\{a_n,b_n\}$ is isomorphic to that of $\mathbf{F}_2$ with the standard marking.
We identify $a_n,b_n$ with homeomorphisms in $F_2'$. 

Given an interval $I=[k_1,k_2]$ with $k_1,k_2\in \mathbf{Z}$, we define $p_{I,n}, q_{I,n}:\mathbf{R}\to \mathbf{R}$
as $$p_{I,n}(t)=
\begin{cases}
 t &\text{ if }t\notin I\\
 t\cdot f_1^{-k} a_n f_1^k&\text{ if }t\in [k,k+1]\subset I, k\in \mathbf{Z}\\
\end{cases}
$$
$$
q_{I,n}(t)=
\begin{cases}
 t&\text{ if }t\notin I\\
 t\cdot f_1^{-k} b_n f_1^k&\text{ if }t\in [k,k+1]\subset I, k\in \mathbf{Z}\\
\end{cases}
$$
Furthermore, we define $$g_{I,n}=p_{I,n} f_1\qquad h_{I,n}=q_{I,n} f_2$$

Let $$U=u_1...u_{l_1}\qquad V=v_1...v_{l_2}$$ be words in $\{f_1,f_2,f_1^{-1},f_2^{-1}\}$ that generate the subgroup $F_2\leq F_1$.
Fix an $m\in \mathbf{N}$ such that for each $1\leq i\leq l_1, 1\leq j\leq l_2$ 
$$([m-1,\infty)\cdot u_1...u_i)\cap [0,2]=\emptyset= ([m-1,\infty)\cdot v_1...v_j)\cap [0,2]$$
and 
$$([0,2]\cdot u_1...u_i)\cap [m-1,\infty)=\emptyset =([0,2]\cdot v_1...v_j)\cap [m-1,\infty)$$

For each $n\in \mathbf{N}$, we fix an interval $I_n=[m,m+2n+1]$.
For simplicity of notation, we denote $g_{I_n,n}, h_{I_n,n}$ as $g_n,h_n$, respectively and $p_{I_n,n},q_{I_n,n}$ as $p_n,q_n$ respectively.

{\bf Claim}: For each $n\in \mathbf{N}$:
\begin{enumerate}
\item The $n$-ball of the Cayley graph of $\langle g_n,h_n\rangle$ is isomorphic to the $n$-ball of $\mathbf{F}_2$ with the standard marking.
\item The elements $g_n,h_n$ generate $F_1$.
\end{enumerate}

The first part of the claim follows immediately from the definitions.
To see this, note that the action of any word in $g_n,h_n$ of length at most $n$ on the interval
$[m+n,m+n+1]$ will mimic the action of the corresponding word in $a_n,b_n$ on $[0,1]$ up to conjugation and post composition by integer translations.

We shall prove the second part.
It suffices to show that $p_n,q_n\in \langle g_n,h_n\rangle$, since $g_n=p_n f_1, h_n=q_nf_2$ and $\langle f_1,f_2\rangle=F_1$.
Note that $p_n,q_n$ can be conjugated inside the group $\langle g_n,h_n\rangle$ (indeed by a power of $h_n$) to elements whose closure of support is properly contained in $(0,1)$,
and hence lie in $F_2'$.
Hence it suffices to show that $F_2'\subset \langle g_n,h_n\rangle$.

Let $U_1,V_1$ be words obtained by replacing occurrences of $f_1,f_2,f_1^{-1},f_2^{-1}$ in $U,V$ respectively by $g_n,h_n,g_n^{-1},h_n^{-1}$.
From our careful choice of $m$, it follows that there is a compact interval $J\subset \mathbf{R}$ with $inf(J)>2$ such that:
\begin{enumerate}
\item For each $w\in \langle U_1,V_1\rangle$, $Supp(w)\subset [0,1]\cup J$.
\item The restrictions $U_1,V_1\restriction [0,1]$ generate $F_2\restriction [0,1]$.
\end{enumerate}
Therefore given an arbitrary commutator $w\in F_2'$, we can find $w_1,w_2\in \langle U_1,V_1\rangle$ such that $[w_1,w_2]\restriction [0,1]=w$.
We will show that $w\in \langle g_n,h_n\rangle$.
Let $\alpha=g_n^{-1}h_n g_n$.
Note that $\alpha\restriction [0,1]=id$ and $t\cdot \alpha>t$ for each $t\in (1,\infty)$.
For a sufficiently large $k\in \mathbf{N}$ it holds that $$w=[\alpha^k w_1 \alpha^{-k}, w_2]$$
Therefore, $w\in \langle g_n,h_n\rangle$.
Since $w$ was an arbitrary commutator in $F_2'$, it follows that $F_2'\subset \langle g_n,h_n\rangle$ and so $F_1=\langle g_n,h_n\rangle$.
Hence the required markings are $(F, \{g_n,h_n\})$ for $n\in \mathbf{N}$.
\end{proof}

\end{document}